\documentclass[11pt]{amsart}
\usepackage{amssymb,amsmath,amsthm}
\usepackage[utf8]{inputenc}
\usepackage[italian,english]{babel}
\usepackage{mathrsfs}
\usepackage{tikz}
\usepackage{enumitem}

\theoremstyle{definition}
\newtheorem{definition}{Definition}[section]
\theoremstyle{plain}
\newtheorem{thm}[definition]{Theorem}

\newtheorem{corollary}[definition]{Corollary}
\newtheorem{prop}[definition]{Proposition} 
\newtheorem{lemma}[definition]{Lemma}

\theoremstyle{remark}

\newtheorem{remark}[definition]{Remark}

\title{A merging procedure for labelings of bipartite graphs}

\author{Paola Bonacini}
\email{paola.bonacini@unict.it}

\author{Lucia Marino}
\email{lucia.marino@unict.it}

\begin{document}
\address{Università degli Studi di Catania\\
  Viale A. Doria 6\\
95125 Catania\\
Italy}


\maketitle

\begin{abstract}
  Let $G$ a bipartite graph with vertex bipartition $\{A,B\}$ and let $m=|E(G)|$. An $(A,B)$-uniformly ordered labeling of $G$ is a labeling $f\colon V\rightarrow [0,2m]$ which, among other conditions, requires that there exists $\lambda\in \mathbb N$ such that $f(a)\le \lambda$ and $f(b)>\lambda$ for all $a\in A$ and $b\in B$. The existence of such a labeling for $G$ implies the existence of a cyclic $G$-decomposition of $K_{2mx+1}$ for all positive integers $x$. In this paper, as a starting point, through this type of labeling we prove the existence of a cyclic $G$-decomposition in the case that $G$ is a cycle of even length with either one or two pendant paths of any length. Then, through a merging procedure, we are able to get this type of labeling for a specific class of bipartite graphs, which are obtained by iteratively adding an even cycle and a pendant path.
\end{abstract}

\section{Introduction}

Let $K_n=(V,E)$ be the complete graph on the vertex set $V$, where $|V|=n$, with $n\in \mathbb N$. Given a subgraph $G$ of $K_n$, a \emph{$G$-design} of order $v$ is a couple $\Sigma=(X,\mathcal B)$, with $|X|=v$, where $\mathcal B$ is a set of disjoint graphs isomorphic to $G$ that decompose $K_v$. The elements of $\mathcal B$ are called \emph{blocks}.  A $G$-design $\Sigma=(X,\mathcal B)$ is called \emph{cyclic} if the automorphism group of $\mathcal B$ contains the cyclic group of order $v$. In such a case, the family $\mathcal B$ of blocks is said to be a \emph{cyclic decomposition} of $K_v$.

Given a graph $G=(V,E)$ without isolated vertices, with $n=|V|$ and $m=|E|$, a labeling of $G$ is usually defined as an injective function $f\colon V\rightarrow \mathbb N$. A labeling $f$ of $G$ induces a function $\tilde{f}\colon E \rightarrow \mathbb N\setminus \{0\}$ defined as $\tilde f(\{u,v\}))=|f(u)-f(v)|$ for any $\{u,v\}\in E$. We call \emph{difference set} of the labeling $f$ the set $\{|f(u)-f(v)|\mid \{u,v\}\in E\}$.

For any $a,b\in \mathbb N$, we denote by $[a,b]$ the set of integers $x\in\mathbb N$ such that $a\le x\le b$. In \cite{R} Rosa called $\rho$-labeling a labeling satisfying the following two conditions:
\begin{enumerate}
  \item $\operatorname{Im} f\subseteq [0,2m]$
        \item $\operatorname{Im} \tilde f=\{x_1,\dots,x_m\}$, where either $x_i=i$ or $x_i=2m+1-i$ for each $i\in [1,m]$.
\end{enumerate}
 Rosa in \cite{R} proved:
\begin{thm}[{\cite[Theorem 7]{R}}]
A cyclic decomposition of the complete graph $K_{2m+1}$ into subgraphs isomorphic to a given graph $G$ with $m$ edges exists if and only if there exists a $\rho$-labeling of the graph $G$.
\end{thm}
Thanks to this result, graph labelings have received a lot of attention since their introduction. The most extensive reference for the subject is the dynamic survey \cite{G}. In this paper we are going to focus our attention on bipartite graphs. If a graph $G$ without isolated vertices is bipartite with vertex bipartition $\{A,B\}$ and $f$ is a labeling of $G$, two more conditions for $f$  can be considered:
\begin{enumerate}
  \item[$(3)$] $f(a)<f(b)$ for each $\{a,b\}\in E$;
        \item[$(4)$] there exists $\lambda\in \mathbb N$ such that $f(a)\le \lambda$ for all $a\in A$ and $f(b)>\lambda$ for all $b\in B$.
\end{enumerate}
A labeling is called $\rho^+$-labeling (or also \emph{ordered} $\rho$-labeling) if it is a $\rho$-labeling satisfying the above condition $(3)$ and it is called \emph{uniformly ordered} $\rho$-labeling or $\rho^{++}$-labeling if it is a $\rho$-labeling satisfying the above condition $(4)$, where obviously a $\rho^{++}$-labeling is a $\rho^+$-labeling. In \cite{EVP} it has been proved the following:
\begin{thm}[{\cite[Theorem 5]{EVP}}]  \label{T:0}
If a bipartite graph $G$ with $m$ edges has a $\rho^+$-labeling and $x$ is any positive integer, then there exists a cyclic $G$-decomposition of $K_{2mx+1}$.
\end{thm}

It is remarkable that, given a $\rho^+$-labeling of a bipartite graph $G$ with $m$ edges, by looking at the proof of Theorem \ref{T:0} it is extremely easy to construct a cyclic decomposition of $K_{2mx+1}$ for any $x>0$. So, when we give a $\rho^+$-labeling of $G$, we are actually providing a specific cyclic decomposition of $K_{2mx+1}$.

In this paper, given a graph $G=(V,E)$ without isolated vertices, in a labeling of $G$ we will specify the codomain and so a labeling is an injective function $f\colon V\rightarrow [0,t]$, with $t\in\mathbb N$. So, the induced function on the edges is $\tilde{f}\colon E\rightarrow [1,t]$ given by $\tilde f(\{u,v\}))=|f(u)-f(v)|$ for any $\{u,v\}\in E$. If $G$ is a bipartite graph without isolated vertices with a vertex bipartition $\{A,B\}$, an \emph{$(A,B,t)$-uniformly ordered labeling} of $G$ is a labeling $f\colon V\rightarrow [0,t]$ such that:
\begin{enumerate}[label={(\alph*)}]
  \item $t\ge 2m$
  \item $\tilde{f}$ is injective
  \item if $\operatorname{Im}(\tilde{f})=\{x_1,\dots,x_m\}$, then it doesn't exist any $i\in\{1,\dots,m\}$ such that both $i$ and $t+1-i\in\operatorname{Im}(\tilde{f})$
  \item condition $(4)$ above holds.
\end{enumerate}
Clearly, if $t=2m$, this labeling trivially corresponds to a $\rho^{++}$-labeling and  $f$ will be called an $(A,B)$-uniformly ordered labeling of $f$, thus simply specifying the sets of vertices receiving the lower and upper labelings. A labeling satisfying conditions $(a)$, $(b)$ and $(c)$ will be called $\overline{\rho}$-labeling. (Note that an $(A,B)$-uniformly ordered labeling is lvactually a uniformly ordered $\rho$-labeling. However, since we want to specify the sets $A$ and $B$ and the labelings in the paper are, in the end, all $\rho$-labelings, we say that these labelings are $(A,B)$-uniformly ordered or just uniformly ordered.)

Let $f\colon V\rightarrow [0,t]$ a labeling of the vertices of a graph $G$ and let $k\in [0,t]$. We denote by $f_k$ the labeling $f_k\colon V\rightarrow [0,t]$ defined by:
\[
  f_k(v)=f(v)+k \quad\forall v\in V,
\]
where the sum is taken $\!\mod (t+1)$. Obviously, $f_0=f$ and $f_k$ is called \emph{$k$-shift} of $f$. Note that $\widetilde{f_k}=\tilde{f}$ for all $k\in [0,t]$.

In this paper, in Section $2$ we prove some general results on shifts of uniformly ordered labelings. In Section $3$ we introduce for bipartite graphs two types of labelings, called of \emph{alternating type}, which either are uniformly ordered or with uniformly ordered shifts. These types of labelings and the technical lemmas in these sections are the key in the merging technique, that we use in Section $4$ and, above all, in Section $5$ and which turns out to be straightforward.

Let $k,m,n_1,\dots,n_k\in\mathbb N$. In this paper we introduce a connected graph, denoted by $C_{m,(n_1,\dots,n_k)}$, with $m+n_1+\dots+n_k$ vertices and $m+n_1+\dots+n_k$ edges, consisting of a cycle $C_m$ and $k$ pendant paths $P_{n_1+1}$, \dots, $P_{n_k+1}$, paths and cycle all having exactly just one vertex in common, that we call the \emph{root} of the graph.

In Theorem \ref{T:1} we prove that $C_{m,(n)}$, for any $m$ even and any $n\in\mathbb N$, admits a uniformly ordered labeling. This result, together with the case $m$ odd considered in \cite{BEOV}, shows the existence of a cyclic $C_{m,(n)}$-decomposition of $K_{2(m+n)x+1}$ for all $m,n,x\in\mathbb N$, with $m\ge 3$ and $x\ge 1$. Moreover, in Theorem \ref{T:3} we also prove that $C_{m,(n_1,n_2)}$, for any $m,n_1,n_2\in\mathbb N$ and $m$ even, admits a uniformly ordered labeling.

Let $t,m_1,\dots,m_t,n_1,\dots,n_t\in\mathbb N$, with $m_1,\dots,m_t\ge 3$. We denote by
\[
  C_{m_1,(n_1),m_2,(n_2),\dots,m_t,(n_t)}
\]
  the graph with $m_1+\dots+m_t+n_1+\dots+n_t-(t-1)$ vertices and $m_1+\dots+m_t+n_1+\dots+n_t$ edges consisting of $t$ cycles of lengths $m_1$,\dots,$m_t$ and $t$ paths $P_{n_1+1}$,\dots, $P_{n_t+1}$ such that:

\begin{itemize}
  \item the cycle $C_{m_1}$ has exactly one vertex in common with the path $P_{n_1+1}$
  \item for any $i=1,\dots,t-1$ the path $P_{n_i+1}$ has an end vertex in common with $C_{m_i}$ and the other end vertex with the cycle $C_{m_{i+1}}$
  \item for any $i=2,\dots,t$ the cycle $C_{m_i}$ has one vertex in common with $P_{n_{i-1}+1}$ and one with $P_{n_i+1}$ and these two vertices have distance $2$
        \item no other vertices are common to cycles and paths.
\end{itemize}
The vertices of this graph that have degree greater than $2$ are called \emph{roots}. See Figure \ref{F:1} for an example.

\begin{figure}[h!] \label{F:1}
\begin{tikzpicture}[scale=0.9]
  \draw (0,0) -- (1,0) -- (2,0) -- (3,0) -- (4,0) -- (4,1) -- (5,2) -- (6,1) -- (7,1) -- (7,0) -- (8,0) -- (9,0) -- (10,0) -- (11,0);
  \draw (0,0) -- (0,1) -- (1,1) -- (1,0);
  \draw (3,0) -- (3,1) -- (4,1) -- (4,0);
  \draw (5,2) -- (6,3) -- (7,3) -- (8,2) -- (7,1);
  \draw (7,1) -- (8,1) -- (8,0);
  \fill (0,0) circle (2pt);
  \fill (0,1) circle (2pt);
  \fill (1,0) circle (2pt);
  \fill (1,1) circle (2pt);
  \fill (2,0) circle (2pt);
  \fill (3,0) circle (2pt);
  \fill (3,1) circle (2pt);
  \fill (4,0) circle (2pt);
  \fill (4,1) circle (2pt);
  \fill (5,2) circle (2pt);
  \fill (6,1) circle (2pt);
  \fill (6,3) circle (2pt);
  \fill (7,0) circle (2pt);
  \fill (7,1) circle (2pt);
  \fill (7,3) circle (2pt);
  \fill (8,0) circle (2pt);
  \fill (8,1) circle (2pt);
  \fill (8,2) circle (2pt);
  \fill (9,0) circle (2pt);
  \fill (10,0) circle (2pt);
  \fill (11,0) circle (2pt);
\end{tikzpicture}
\caption{The graph $C_{4,(2),4,(1),6,(0),(4),3}$}
\end{figure}
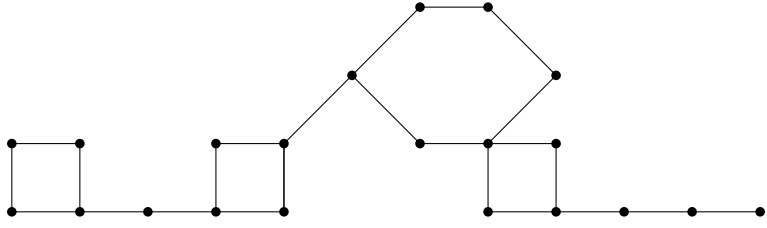

In Theorem \ref{T:5} we prove the existence of a uniformly ordered labeling  for the graph $C_{m_1,(n_1),m_2,(n_2),\dots,m_t,(n_t)}$ in the case that $m_i\equiv 0\mod 4$ for $i=1,\dots,t-1$ and $n_1,\dots,n_{t-1},m_t$ even. Theorem \ref{T:5} partially extends an analogous result obtained in \cite{BBCEPSZ} in which, through various types of labelings, it has been proved the existence of a cyclic $C_{m_1,(0),m_2,(0)}$-decomposition of $K_{2(m_1+m_2)x+1}$ for any $m_1,m_2,x\in\mathbb N$, with $m_1,m_2\ge 3$ and $x\ge 1$.

\section{Uniformly ordered labelings}

In this section we will state some lemmas on labelings of bipartite graphs. In this paper, all the graphs have no isolated vertices.

\begin{lemma}  \label{L:1}
  Let $G$ be a bipartite graph with vertex bipartition $\{A,B\}$, where $|A|,|B|\ge 2$. Let $f\colon A\cup B\rightarrow [0,t]$ be a $\overline{\rho}$-labeling of the vertices of $G$, with $t\in\mathbb N$. Then there exist $k\in [0,t]$ such that $f_k$ is $(A,B,t)$ or $(B,A,t)$-uniformly ordered if and only if either $[\min f(A),\max f(A)]\cap f(B)=\emptyset$ or $[\min f(B),\max f(B)]\cap f(A)=\emptyset.$
\end{lemma}
\begin{proof}
 Suppose that at the same time  $[\min f(A),\max f(A)]\cap f(B)\neq \emptyset$ and $[\min f(B),\max f(B)]\cap f(A)\neq \emptyset$. This implies that there exist $a_1,a_2\in A$ and $b_1,b_2\in B$ such that either $f(a_1)<f(b_1)<f(a_2)<f(b_2)$ or $f(b_1)<f(a_1)<f(b_2)<f(a_2).$ Suppose that $f(a_1)<f(b_1)<f(a_2)<f(b_2)$. Then, since for $f(a_1)>0$:
 \begin{multline*}
  [0,t]=[0,t-f(b_2)]\cup [t+1-f(b_2),t-f(a_2)]\cup\\
  \cup [t+1-f(a_2),t-f(b_1)]\cup [t+1-f(b_1),t-f(a_1)]\cup\\
  \cup [t+1-f(a_1),t],
\end{multline*}
and for $f(a_1)=0$:
 \begin{multline*}
  [0,t]=[0,t-f(b_2)]\cup [t+1-f(b_2),t-f(a_2)]\cup\\
  \cup [t+1-f(a_2),t-f(b_1)]\cup [t+1-f(b_1),t],
\end{multline*}
in both cases it follows that there isn't any $k\in[0,t]$ such that either $f_k$ is $(A,B,t)$ or $(B,A,t)$-uniformly ordered.

Conversely, suppose that  $[\min f(A),\max f(A)]\cap f(B)=\emptyset$ and that $f$ is not $(A,B,t)$ or $(B,A,t)$-uniformly ordered, because otherwise the statement holds for $k=0$. So, if $B=\{b_1,\dots,b_s\}$ with $f(b_1)<\dots<f(b_s)$, there exists $p\in [1,s-1]$ such that $f(b_p)<f(a)<f(b_{p+1})$ for any $a\in A$. Then, it is easy to see that for $k\in [t+1- f(b_{p+1}),t-\max f(A)]$ $f_k$ is $(B,A,t)$-uniformly ordered and for $k\in [t+1-\min f(A),t-f(b_p)]$ $f_k$ is $(A,B,t)$-uniformly ordered.
\end{proof}

By the proof of the Lemma \ref{L:1} the following is immediate:
\begin{corollary}  \label{C:1}
  Let $G$ be a bipartite graph with vertex bipartition $\{A,B\}$, where $|A|,|B|\ge 2$. Let $f\colon A\cup B\rightarrow [0,t]$ be a non uniformly ordered $\overline{\rho}$-labeling of the vertices of $G$, with $t\in\mathbb N$, such that $[\min f(A),\max f(A)]\cap f(B)=\emptyset$. Let $x=\max \{f(b)\mid f(b)<f(a)\, \forall a\in A\} $ and $y=\min \{f(b)\mid f(b)>f(a)\,\forall a\in A\}.$ Then:

  \begin{enumerate}
        \item the labeling $f_k$ is a $(A,B,t)$-uniformly ordered labeling of $G$ if and only if $k\in [t+1-\min f(A),t-x]$;
    \item  the labeling $f_k$ is a $(B,A,t)$-uniformly ordered labeling of $G$ if and only if $k\in [t+1-y,t-\max f(A)]$.
   \end{enumerate}
\end{corollary}

It is also easy to determine when the shift of a uniformly ordered labeling is uniformly ordered.
\begin{lemma}  \label{L:2}
  Let $G$ be a bipartite graph with vertex bipartition $\{A,B\}$, where $|A|,|B|\ge 2$. Let $f\colon A\cup B\rightarrow [0,t]$ be an $(A,B,t)$-uniformly ordered labeling of $G$, with $t\in\mathbb N$. Then:
\begin{enumerate}
  \item the labeling $f_k$ is a $(A,B,t)$-uniformly ordered labeling of $G$ if and only if $k\in [0,t-\max f(B)] \cup [t+1-\min f(A),t]$, in the case that $\min f(A)\ge 1$, or $k\in [0,t-\max f(B)]$, in the case that $\min f(A)=0$;
    \item the labeling $f_k$ is a $(B,A,t)$-uniformly ordered labeling of $G$ if and only if $k\in [t+1-\min f(B),t-\max f(A)]$.
  \end{enumerate}
\end{lemma}
\begin{proof}
  Let $p_1=\min f(A)$, $p_2=\max f(A)$, $q_1=\min f(B)$ and $q_2=\max f(B).$ Then by the definition of $(A,B,t)$-uniformly ordered we have $p_1<p_2<q_1<q_2$. Then the statement follows easily.
\end{proof}

\section{Labelings of alternating type}

In this section we state some key lemmas for the main results of this paper. We start with the following, whose proof is immediate.

\begin{lemma} \label{L:3}
  Let $r,s\in\mathbb N$, with $r\ge 2s+1$, and let $a_1,\dots,a_{2s}\in [1,r-1]$ such that $a_i<a_{i+1}$ for $i=1.\dots,2s-1$. Let:
\[
  x_i=\begin{cases}
      a_1 & \text{for }i=1\\[2ex]
      \displaystyle{\sum_{j=0}^{i-1}a_{2j+1} - \sum_{j=1}^{i-1}a_{2j}} & \text{for }i=2,\dots,s
    \end{cases}
\]
and 
  \[
  y_i=\sum_{j=0}^{i-1} a_{2j+1}-\sum_{j=1}^i a_{2j}+r,
\]
for $i=1,\dots,s$. Then $0<x_1<\dots<x_s<y_s<\dots<y_1<r$ and the residue classes of $0,x_1,\dots,x_s,y_1,\dots,y_s\in\mathbb Z_r$ are pairwise different. 
\end{lemma}

\begin{remark}
  Let $r,s\in\mathbb N$, with $r\ge 2s+1$, and $a_1,\dots,a_{2s}$ such that $a_i<a_{i+1}$ for $i=1,\dots,s-1$. Let $b_i=\min\{a_i,r-a_i\}$ for $i=1,\dots,2s$ and suppose that $b_i\ne b_j$ for any $i\ne j$. Let $x_1,\dots,x_s,y_1,\dots,y_s$ defined as above. Then in the path:
\[
  <0,x_1,y_1,x_2,y_2,\dots,y_{s-1},x_s,y_s>
\]
  the differences in $\mathbb Z_r$ are exactly $b_1,\dots,b_{2s}$.
\end{remark}

\begin{corollary} \label{C:1}
  Let $r,s\in \mathbb N$, with $r\ge 2s+1$, and let $a_1,\dots,a_{2s}\in\mathbb [1,r-1]$ such that $a_i<a_{i+1}$ for $i=1,\dots,2s-1$.  Let:
\[
  x_i=\begin{cases}
      a_1 & \text{for }i=1\\[2ex]
      \displaystyle{\sum_{j=0}^{i-1}a_{2j+1} - \sum_{j=1}^{i-1}a_{2j}} & \text{for }i=2,\dots,s
    \end{cases}
\]
and 
  \[
  y_i=\sum_{j=0}^{i-1} a_{2j+1}-\sum_{j=1}^i a_{2j}+r,
\]
for $i=1,\dots,s$. Let $G$ be a bipartite graph with vertex bipartition $\{A,B\}$ such that $|A|=s$ and $|B|=s+1$ and let $f\colon A\cup B\rightarrow [0,r-1]$ be a $\overline{\rho}$-labeling such that  $f(A)=\{x_1-1,\dots,x_s-1\}$ and $f(B)=\{y_1-1,\dots,y_s-1,r-1\}$. Then:
\begin{enumerate}
  \item $f_k$ is a $(A,B,r-1)$-uniformly ordered labeling of $G$ if and only if $k\in \{0\}\cup [r+1-a_1,r-1]$ for $a_1\ge 2$, and $k=0$ for $a_1=1$;
  \item $f_{k}$ is $(B,A,r-1)$-uniformly ordered if and only if $k\in[r+1-y_s,r-x_s]$.
\end{enumerate}
\end{corollary}
\begin{proof}
  The proof follows by Lemma \ref{L:2} and Lemma \ref{L:3}. 
\end{proof}

It is immediate the following:
\begin{corollary} \label{C:2}
  Let $r,s\in \mathbb N$, with $r\ge 2s$, and let $a_1,\dots,a_{2s-1}\in\mathbb [1,r-1]$ such that $a_i<a_{i+1}$ for $i=1,\dots,2s-2$, if $s\ge 2$.  Let:
\[
  x_i=\begin{cases}
      a_1 & \text{for }i=1\\[2ex]
      \displaystyle{\sum_{j=0}^{i-1}a_{2j+1} - \sum_{j=1}^{i-1}a_{2j}} & \text{for }i=2,\dots,s, \text{ if }s\ge 2
    \end{cases}
\]
and for $s\ge 2$
  \[
  y_i=\sum_{j=0}^{i-1} a_{2j+1}-\sum_{j=1}^i a_{2j}+r,
\]
for $i=1,\dots,s-1$. Let $G$ be a bipartite graph with vertex bipartition $\{A,B\}$ such that $|A|=s$ and $|B|=s$ and let $f\colon A\cup B\rightarrow [0,r-1]$ be a $\overline{\rho}$-labeling such that $f(A)=\{x_1-1,\dots,x_s-1\}$ and $f(B)=\{y_1-1,\dots,y_{s-1}-1,r-1\}$. Then:
\begin{enumerate}
  \item $f_{k}$ is a $(A,B,r-1)$-uniformly ordered labeling of $G$ if and only if $k\in \{0\}\cup [r+1-a_1,r-1]$ for $a_1\ge 2$, and $k=0$ for $a_1=1$;
  \item $f_{k}$ is $(B,A,r-1)$-uniformly ordered if and only if $k\in[r+1-y_{s-1},r-x_s]$.
\end{enumerate}
\end{corollary}

Let $G$ be a bipartite graph and let $r,n\in\mathbb N$, with $r\ge n+1$. Let $a_1,\dots,a_{n}\in  [1,r-1]$ such that $a_1<\dots<a_{n}$. Let:
\[
\underline a= \begin{cases}
  (a_1,-a_2,a_3,\dots,a_{n-1},-a_n) & \text{ for $n$ even}\\
  (a_1,-a_2,a_3,\dots,-a_{n-1},a_n)  & \text{ for $n$ odd}.
\end{cases}
\]
A labeling $f$ as either in Corollary \ref{C:1} (so the labelings are $r-1,x_1-1,y_1-1,\dots,x_s-1,y_s-1$) or Corollary \ref{C:2} (so the labelings are $r-1,x_1-1,y_1-1,\dots,x_{s-1}-1,y_{s-1}-1,x_s-1$) is said to be \emph{an alternating labeling of type $\underline a$} and its $k$-shift labelings are said to be of type $(\underline a,k)$ for any $k\in\mathbb N$.

In what follows, we state some results, analogous to those given above.

\begin{lemma} \label{L:4}
  Let $r,s\in\mathbb N$, with $r\ge 2s+1$, and let $a_1,\dots,a_{2s}\in [1,r-1]$ such that $a_i<a_{i+1}$ for $i=1.\dots,2s-1$. Let:
\[
  x_i=  \displaystyle{\sum_{j=1}^{i}a_{2j} - \sum_{j=0}^{i-1}a_{2j+1}}  \quad \text{for }i=1,\dots,s
\]
and
  \[
    y_i= \begin{cases}
      r-a_1 & \text{for }i=1\\[2ex]
      r+\displaystyle{\sum_{j=1}^{i-1}a_{2j} - \sum_{j=0}^{i-1}a_{2j+1}} & \text{for }i=2,\dots,s.
    \end{cases}
\]
Then $0<x_1<\dots<x_s<y_s<\dots<y_1<r$ and the residue classes of $0,x_1,\dots,x_s,y_1,\dots,y_s\in\mathbb Z_r$ are pairwise different.
\end{lemma}

\begin{remark}
  Let $r,s\in\mathbb N$, with $r\ge 2s+1$, and $a_1,\dots,a_{2s}$ such that $a_i<a_{i+1}$ for $i=1,\dots,s-1$. Let $b_i=\min\{a_i,r-a_i\}$ for $i=1,\dots,2s$ and suppose that $b_i\ne b_j$ for any $i\ne j$. Let $x_1,\dots,x_s,y_1,\dots,y_s$ defined as in Lemma \ref{L:4}. Then in the path:
\[
  <0,y_1,x_1,y_2,x_2,\dots,x_{s-1},y_s,x_s>
\]
  the differences in $\mathbb Z_r$ are exactly $b_1,\dots,b_{2s}$.
\end{remark}

We immediately get the following two corollaries, analogous to Corollary \ref{C:1} and Corollary \ref{C:2}:

\begin{corollary} \label{C:3}
  Let $r,s\in \mathbb N$, with $r\ge 2s+1$, and let $a_1,\dots,a_{2s}\in\mathbb [1,r-1]$ such that $a_i<a_{i+1}$ for $i=1,\dots,2s-1$.  Let:
\[
  x_i=  \displaystyle{\sum_{j=1}^{i}a_{2j} - \sum_{j=0}^{i-1}a_{2j+1}}  \quad \text{for }i=1,\dots,s
\]
and
  \[
    y_i= \begin{cases}
      r-a_1 & \text{for }i=1\\[2ex]
      r+\displaystyle{\sum_{j=1}^{i-1}a_{2j} - \sum_{j=0}^{i-1}a_{2j+1}} & \text{for }i=2,\dots,s.
    \end{cases}
\]
Let $G$ be a bipartite graph with vertex bipartition $\{A,B\}$ such that $|A|=s+1$ and $|B|=s$ and let $f\colon A\cup B\rightarrow [0,r-1]$ be a $\overline{\rho}$-labeling such that  $f(A)=\{0,x_1,\dots,x_s\}$ and $f(B)=\{y_1,\dots,y_s\}$. Then:
\begin{enumerate}
  \item $f_k$ is a $(A,B,r-1)$-uniformly ordered labeling of $G$ if and only if $k\in [0,a_1-1]$;
  \item $f_{k}$ is $(B,A,r-1)$-uniformly ordered if and only if $k\in[r-y_s,r-1-x_s]$.
\end{enumerate}
\end{corollary}

\begin{corollary} \label{C:4}
  Let $r,s\in \mathbb N$, with $r\ge 2s$, and let $a_1,\dots,a_{2s-1}\in\mathbb [1,r-1]$ such that $a_i<a_{i+1}$ for $i=1,\dots,2s-2$, if $s\ge 2$.  Let:
\[
  x_i=  \displaystyle{\sum_{j=1}^{i}a_{2j} - \sum_{j=0}^{i-1}a_{2j+1}}  \quad \text{for }i=1,\dots,s-1,  \text{ if }s\ge 2
\]
and
  \[
    y_i= \begin{cases}
      r-a_1 & \text{for }i=1\\[2ex]
      r+\displaystyle{\sum_{j=1}^{i-1}a_{2j} - \sum_{j=0}^{i-1}a_{2j+1}} & \text{for }i=2,\dots,s, \text{ if }s\ge 2.
    \end{cases}
\]
Let $G$ be a bipartite graph with vertex bipartition $\{A,B\}$ such that $|A|=s$ and $|B|=s$ and let $f\colon A\cup B\rightarrow [0,r-1]$ be a $\overline{\rho}$-labeling such that  $f(A)=\{0,x_1,\dots,x_{s-1}\}$ and $f(B)=\{y_1,\dots,y_s\}$. Then:
\begin{enumerate}
  \item $f_k$ is a $(A,B,r-1)$-uniformly ordered labeling of $G$ if and only if $k\in [0,a_1-1]$;
  \item $f_{k}$ is $(B,A,r-1)$-uniformly ordered if and only if $k\in[r-y_s,r-1-x_{s-1}]$.
\end{enumerate}
\end{corollary}

Let $G$ be a bipartite graph end let $r,n\in\mathbb N$, with $r\ge n+1$. Let $a_1,\dots,a_{n}\in  [1,r-1]$ such that $a_1<\dots<a_{n}$. Let:
\[
\underline a= \begin{cases}
  (-a_1,a_2,-a_3,\dots,-a_{n-1},a_n) & \text{ for $n$ even}\\
  (-a_1,a_2,-a_3,\dots,a_{n-1},-a_n)  & \text{ for $n$ odd}.
\end{cases}
\]
A labeling $f$ as either in Corollary \ref{C:3} or Corollary \ref{C:4} is said to be \emph{an alternating labeling of type $\underline a$} and its $k$-shift labelings are said to be of type $(\underline a,k)$ for any $k\in\mathbb N$.

\begin{remark} \label{R:1}
Note that, given a labeling $f$ of alternating type (of any of the two types defined in this section) of a bipartite graph, we can replace $r$ with any integer $r'>r$ (in the definition of the $y_i$, too) and we get a labeling $f'$ of the same alternating type of $f$. Moreover, if $f$ is $(A,B,r-1)$-uniformly ordered, then $f'$ is  $(A,B,r'-1)$-uniformly ordered.
\end{remark}

\section{Unicyclic graph designs}

In this section we are going to provide uniformly ordered labelings for the graphs $C_{m,(n)}$ and $C_{m,(n_1,n_2)}$ for any $m,n,n_1,n_2\in\mathbb N$, with $m$ even.

\begin{prop}\label{P:1}
  Let $m\in\mathbb N$, with $m\equiv 0\mod 4$ and $m\ge 4$, and let $p,r\in\mathbb N$, with either $r\ge 2p+2m+1$ or $p+m+1\le r\le 2p+1$. Let:
  \[
    a_j=\begin{cases}
      p+j & \text{for } j=1,\dots,\dfrac{m}{2}-1\\
      p+j+1 & \text{for } j=\dfrac{m}{2},\dots,m-1.
      \end{cases}
  \]
  Then there exists, for a bipartition $\{A,B\}$, an $(A,B,r-1)$-uniformly ordered labeling $f\colon V(C_m)\rightarrow [0,r-1]$ of the vertices of the cycle $C_m$ of alternating type $\underline{a}=(a_1,-a_2,\dots,-a_{m-2},a_{m-1})$ with difference set $[p+1,p+m]$ and labelings
\[
  f(A)=\{p+i-1\mid i=1,\dots,\tfrac{m}{2}\}
\]
and
\[
  f(B)=\{r-i-1\mid i=1,\dots,\tfrac{m}{4}-1\}\cup\{r-i-2\mid i=\tfrac{m}{4},\dots,\tfrac{m}{2}-1\}\cup\{r-1\}.
  \]
\end{prop}
\begin{proof}
  It is sufficient to give a labeling to the vertices of $C_m$ as in Corollary \ref{C:2}, with $2s=m$, in such a way that two consecutive labelings in the sequence $r-1,x_1-1,y_1-1,\dots,x_{m/2-1}-1,y_{m/2-1}-1,x_{m/2}-1$ are assigned to adjacent vertices. Clearly, this implies that the vertices with labelings $r-1$ and $x_{m/2}-1$ are adjacent and the statement follows by noting that:
\[
  x_i=p+i\quad \text{for }i=1,\dots,\dfrac{m}{2}
\]
and
\[
  y_i= \begin{cases}
    r-i & \text{ for }i=1,\dots,\dfrac{m}{4}-1\\[2ex]
    r-i-1 & \text{ for }i=\dfrac{m}{4},\dots,\dfrac{m}{2}-1,
  \end{cases}
\]
  since the other conditions are immediate.
\end{proof}

In the case $m\equiv 2\mod 4$ we have a slightly different result:
\begin{prop}\label{P:2}
  Let $m\in\mathbb N$, with $m\equiv 2\mod 4$ and $m\ge 6$, and let $p,r\in\mathbb N$, with either $r\ge 2p+2m+1$ or $p+m+2\le r\le 2p+1$. Let:
  \[
    a_j=\begin{cases}
      p+j & \text{for } j=1,\dots,\dfrac{m}{2}\\
      p+j+1 & \text{for } j=\dfrac{m}{2}+1,\dots,m-2\\
      p+j+2 & \text{for } j=m-1.
      \end{cases}
  \]
  Then there exists, for a bipartition $\{A,B\}$, an $(A,B,r-1)$-uniformly ordered labeling $f\colon V(C_m)\rightarrow [0,r-1]$ of the vertices of the cycle $C_m$ of type $\underline{a}=(a_1,-a_2,\dots,-a_{m-2},a_{m-1})$ with difference set $[p+1,p+m-1]\cup \{p+m+1\}$ and labelings
\[
  f(A)=\{p+i-1\mid i=1,\dots,\tfrac{m}{2}-1\}\cup\{p+\tfrac{m}{2}\}
\]
and
\[
  f(B)=\{r-i-1\mid i=1,\dots,\tfrac{m-2}{4}\}\cup\{r-i-2\mid i=\tfrac{m+2}{4},\dots,\tfrac{m}{2}-1\}\cup\{r-1\}.
  \]
\end{prop}
\begin{proof}
  It is sufficient to give a labeling to the vertices of $C_m$ as in Corollary \ref{C:2}, with $2s=m$, in such a way that two consecutive labelings in the sequence $r-1,x_1-1,y_1-1,\dots,x_{m/2-1}-1,y_{m/2-1}-1,x_{m/2}-1$ are assigned to adjacent vertices. Clearly, this implies that the vertices with labelings $r-1$ and $x_{m/2}-1$ are adjacent and the statement follows by noting that
\[
  x_i=\begin{cases}
    p+i & \text{for }i=1,\dots,\dfrac{m}{2}-1\\[2ex]
    p+\dfrac{m}{2}+1 & \text{for }i=\dfrac{m}{2}
    \end{cases}
\]
and
\[
  y_i= \begin{cases}
    r-i & \text{ for }i=1,\dots,\dfrac{m-2}{4}\\[2ex]
    r-i-1 & \text{ for }i=\dfrac{m+2}{4},\dots,\dfrac{m}{2}-1,
  \end{cases}
\]
  since the other conditions are immediate.
\end{proof}

Now we are ready to prove the main results of this section.

\begin{thm} \label{T:1}
Let $m,n\in\mathbb N$, with $m\ge 4$ even. Then there exists a uniformly ordered labeling of $C_{m,(n)}$.
\end{thm}
\begin{proof}
  If $n=0$, we get the statement by Proposition \ref{P:1} and Proposition \ref{P:2} by taking $p=0$ and $r=2m+1=2|E(C_m)|+1$.

  Let $n\ge 1$ and suppose that $m\equiv 0\mod 4$. Let $r=2(m+n)+1=2|E(C_{m,(n)})|+1$ and let:
    \[
    a_j=\begin{cases}
      j & \text{for } j=1,\dots,\dfrac{m}{2}-1\\
      j+1 & \text{for } j= \dfrac{m}{2},\dots,m+n-1.
      \end{cases}
  \]
  Let us define a labeling on $C_{m,(n)}$ as in Corollary \ref{C:1} if $n$ is odd and as in Corollary \ref{C:2} if $n$ is even, in such a way that, given $s=\left\lfloor\tfrac{m+n}{2}\right\rfloor$:
\begin{itemize}
  \item the labelings $r-1$, $x_1-1$, $y_1-1$,\dots, $x_{m/2-1}-1$, $y_{m/2-1}-1$, $x_{m/2}-1$ are given, in this order, to the vertices of the cycle $C_m$;
  \item the root receives the label $x_{m/2}-1$;
  \item the remaining labels, i.e. $y_{m / 2}-1$, $x_{m/2+1}-1$, $y_{m/2+1}-1$,\dots, $x_{s}-1$, $y_{s}-1$ for $n$ odd and $y_{m/2}-1$, $x_{m/2+1}-1$, $y_{m/2+1}-1$,\dots, $x_{s-1}-1$, $y_{s-1}-1$, $x_{s}-1$ for $n$ even, are given in this order to the vertices of the pendant path $P_{n+1}$ starting from the vertex adjacent to the root.
\end{itemize}
The statement follows by Corollary \ref{C:1}, Corollary \ref{C:2} and Proposition \ref{P:1}.

  Let $n\ge 1$ and suppose that $m\equiv 2\mod 4$. Let $r=2(m+n)+1$ and let:
  \[
    a_j=\begin{cases}
      j & \text{for } j=1,\dots,\dfrac{m}{2}\\
      j+1 & \text{for } j=\dfrac{m}{2}+1,\dots,m-2\\
      j+2 & \text{for } j=m-1,\dots,m+n-2\\
      r-m & \text{for } j=m+n-1.
      \end{cases}
  \]
  Let us define a labeling on $C_{m,(n)}$ as in Corollary \ref{C:1} if $n$ is odd and as in Corollary \ref{C:2} if $n$ is even, in such a way that, given $s=\left\lfloor\tfrac{m+n}{2}\right\rfloor$:
\begin{itemize}
  \item the labelings $r-1$, $x_1-1$, $y_1-1$,\dots, $x_{m/2-1}-1$, $y_{m/2-1}-1$, $x_{m/2}-1$ are given, in this order, to the vertices of the cycle $C_m$;
  \item the root receives the label $x_{m/2}-1$;
  \item the remaining labels, i.e. $y_{m/2}-1$, $x_{m/2+1}-1$, $y_{m/2+1}-1$,\dots, $x_{s}-1$, $y_{s}-1$ for $n$ odd and $y_{m/2}-1$, $x_{m/2+1}-1$, $y_{m/2+1}-1$,\dots, $x_{s-1}-1$, $y_{s-1}-1$, $x_{s}-1$ for $n$ even, are given in this order to the vertices of the pendant path $P_{n+1}$ starting from the vertex adjacent to the root.
\end{itemize}
The statement follows by Corollary \ref{C:1}, Corollary \ref{C:2} and Proposition \ref{P:2}.

\end{proof}

By Theorem~\ref{T:1} and Theorem \ref{T:0} we immediately get:

\begin{corollary}
 For any $m,n\in\mathbb N$, with $m\ge 4$ even, there exists a cyclic $C_{m,(n)}$-decomposition of $K_v$ for any $v\equiv 1 \mod 2(m+n)$.
\end{corollary}

We can also prove the following:

\begin{thm}\label{T:3}
Let $m,n_1,n_2\in\mathbb N$, with $m\ge 4$ even. Then there exists a uniformly ordered labeling of $C_{m,(n_1,n_2)}$.
\end{thm}
\begin{proof}
  Let $n_1,n_2\ge 1$, since otherwise the statement follows by Theorem \ref{T:1}. We can also suppose that $n_1\le n_2$.

  \textbf{Case 1.} Suppose that $m\equiv 0\mod 4$ and let $n_1\le n_2$. Let $r=2(m+n_1+n_2)+1$  and let:
    \[
    a_j=\begin{cases}
      j & \text{for } j=1,\dots,\dfrac{m}{2}-1\\
      j+1 & \text{for } j= \dfrac{m}{2},\dots,m+n_1-1.
      \end{cases}
    \]
  Let $C_{m,(n_1)}$ be the graph obtained by $C_{m,(n_1,n_2)}$ by deleting the pendant path $P_{n_2+1}$. Let us define a labeling on $C_{m,(n_1)}$ in a way to the similar to the one of Theorem \ref{T:1} (so the numbers $x_i$ and $y_i$ are defined as in Corollary \ref{C:1} and Corollary \ref{C:2}). So, given $s=\left\lfloor\tfrac{m+n_1}{2}\right\rfloor$:
\begin{itemize}
  \item the labelings $r-1$, $x_1-1$, $y_1-1$,\dots, $x_{m/2-1}-1$, $y_{m/2-1}-1$, $x_{m/2}-1$ are given, in this order, to the vertices of the cycle $C_m$;
  \item the root receives the label $x_{m/2}-1$;
  \item the remaining labels, i.e. $y_{m/2}-1$, $x_{m/2+1}-1$, $y_{m/2+1}-1$,\dots, $x_{s}-1$, $y_{s}-1$ for $n$ odd and $y_{m/2}-1$, $x_{m/2+1}-1$, $y_{m/2+1}-1$,\dots, $x_{s-1}-1$, $y_{s-1}-1$, $x_{s}-1$ for $n$ even, are given in this order to the vertices of the pendant path $P_{n_1+1}$ starting from the vertex adjacent to the root.
\end{itemize}
Note that for $i\ge \tfrac{m}{2}$:
\[
  x_i=i \quad\text{ and }\quad y_i=r-i-1.
\]
This labeling is given to the corresponding vertices of the graph $C_{m,(n_1,n_2)}$. For the remaining vertices of the pendant path $P_{n_2+1}$, let:
\[
a'_i= \begin{cases}
  m+n_1+1 & \text{for }i=1\\
  r-(m+n_1+n_2+2-i) & \text{for }i=2,\dots,n_2, \text{ if }n_2\ge 2.
\end{cases}
\]
Let us define a labeling of alternating type $((-a'_1,a'_2,-a'_3,\dots),\tfrac{m}{2}-1)$ on the pendant path $P_{n_2+1}$ as done in Corollary \ref{C:3} and Corollary \ref{C:4} in such a way that the labels are
\begin{equation} \label{eq:1}
  \dfrac{m}{2}-1,y'_1+\dfrac{m}{2}-1,x'_1+\dfrac{m}{2}-1,y'_2+\dfrac{m}{2}-1,x'_2+\dfrac{m}{2}-1,\dots
\end{equation}
with:
\[
x'_i=\sum_{j=1}^i a'_{2j}-\sum_{j=0}^{i-1} a'_{2j+1}  \quad\text{for }i\ge 1 \text{ and }n_2\ge 2
\]
and
\[
  y'_i= \begin{cases}
  \displaystyle{r-a'_1} & \text{for }i=1\\[2ex]
  \displaystyle{r+\sum_{j=1}^{i-1} a'_{2j}-\sum_{j=0}^{i-1} a'_{2j+1}} & \text{for }i\ge 2.
  \end{cases}
\]
Obviously, $\tfrac{m}{2}-1$ is still the label given to the root of the graph $C_{m,(n_1,n_2)}$ and the remaining labels in \eqref{eq:1} are given, in that order, to the other vertices of the pendant path $P_{n_2+1}$. The statement in this case follows by Corollary \ref{C:3} and Corollary \ref{C:4} and by the following facts, easy to verify:
\[
x'_1+\dfrac{m}{2}-1>x_{s}-1 \text{ in the case }n_2\ge 2,
\]
\[
y'_1+\dfrac{m}{2}-1<y_{s}-1 \quad\text{for $n_1$ odd}
\]
and
\[
y'_1+\dfrac{m}{2}-1<y_{s-1}-1 \quad\text{for $n_1$ even.}
\]
Note that by Corollary \ref{C:3} and Corollary \ref{C:4} this labeling  on the path $P_{n_2+1}$ is $(A',B',r-1)$-uniformly ordered, where $\{A',B'\}$ is the vertex bipartition of $P_{n_2+1}$ such that the root is an element of $A'$. This proves the statement in the case $m\equiv 0\mod 4$.

\textbf{Case 2.} Suppose that $m\equiv 2\mod 4$ and let $n_1\le n_2$. Let $r=2(m+n_1+n_2)+1$  and let:
    \[
    a_j=\begin{cases}
      j & \text{for } j=1,\dots,\dfrac{m}{2}\\
      j+1 & \text{for } j= \dfrac{m}{2}+1\dots,m-2\\
      j+2 & \text{for } j=m-1,\dots,m+n_1-1.
      \end{cases}
    \]
  Let $C_{m,(n_1)}$ be the graph obtained by $C_{m,(n_1,n_2)}$ by deleting the pendant path $P_{n_2+1}$. Let us define a labeling on $C_{m,(n_1)}$ in a way to the similar to the one of Theorem \ref{T:1} (so the numbers $x_i$ and $y_i$ are defined as in Corollary \ref{C:1} and Corollary \ref{C:2}). So, given $s=\left\lfloor\tfrac{m+n_1}{2}\right\rfloor$:
\begin{itemize}
  \item the labelings $r-1$, $x_1-1$, $y_1-1$,\dots, $x_{m/2-1}-1$, $y_{m/2-1}-1$, $x_{m/2}-1$ are given, in this order, to the vertices of the cycle $C_m$;
  \item the root receives the label $x_{m/2}-1$;
  \item the remaining labels, i.e. $y_{m/2}-1$, $x_{m/2+1}-1$, $y_{m/2+1}-1$,\dots, $x_{s}-1$, $y_{s}-1$ for $n$ odd and $y_{m/2}-1$, $x_{m/2+1}-1$, $y_{m/2+1}-1$,\dots, $x_{s-1}-1$, $y_{s-1}-1$, $x_{s}-1$ for $n$ even, are given in this order to the vertices of the pendant path $P_{n_1+1}$ starting from the vertex adjacent to the root.
\end{itemize}
Note that for $i\ge \tfrac{m}{2}$:
\[
x_i=i+1 \quad \text{and} \quad y_i=r-i-1.
\]
This labeling is given to the corresponding vertices of the graph $C_{m,(n_1,n_2)}$. For the remaining vertices of the pendant path $P_{n_2+1}$, let us consider the following cases:
\begin{itemize}
  \item for $n_2=1$ (and so $n_1=1$) let $a'_1=r-m$;
  \item for $n_2=2$ let $a'_1=m+n_1+2$ and $a'_2=r-m$;
        \item for $n_2\ge 3$ let:
       \[
a'_i= \begin{cases}
  m+n_1+2 & \text{for }i=1\\
  r-(m+n_1+n_2+2-i) & \text{for }i=2,\dots,n_2-1\\
  r-m & \text{for }i=n_2.
\end{cases}
\]
\end{itemize}
We define a labeling of alternating type $((-a'_1,a'_2,-a'_3,\dots),\tfrac{m}{2})$ on the pendant path $P_{n_2+1}$ as done in Corollary \ref{C:3} and Corollary \ref{C:4} in such a way that the labels are
\begin{equation} \label{eq:2}
  \dfrac{m}{2},y'_1+\dfrac{m}{2},x'_1+\dfrac{m}{2},y'_2+\dfrac{m}{2},x'_2+\dfrac{m}{2},\dots
\end{equation}
with:
\[
x'_i=\sum_{j=1}^i a'_{2j}-\sum_{j=0}^{i-1} a'_{2j+1}  \quad\text{for }i\ge 1\text{ and if }n_2\ge 2
\]
and
\[
  y'_i= \begin{cases}
  \displaystyle{r-a'_1} & \text{for }i=1\\[2ex]
  \displaystyle{r+\sum_{j=1}^{i-1} a'_{2j}-\sum_{j=0}^{i-1} a'_{2j+1}} & \text{for }i\ge 2.
  \end{cases}
\]
Obviously, $\tfrac{m}{2}$ is still the label given to the root of the graph $C_{m,(n_1,n_2)}$ and the remaining labels in \eqref{eq:2} are given, in that order, to the other vertices of the pendant path $P_{n_2+1}$. The statement in this case follows by Corollary \ref{C:3} and Corollary \ref{C:4} and by the following facts, easy to verify:
\[
x'_1+\dfrac{m}{2}>x_{s}-1 \text{ in the case }n_2\ge 2,
\]
\[
y'_1+\dfrac{m}{2}<y_{s}-1 \quad\text{for $n_1$ odd}
\]
and
\[
y'_1+\dfrac{m}{2}<y_{s-1}-1 \quad\text{for $n_1$ even.}
\]
Note that by Corollary \ref{C:3} and Corollary \ref{C:4} this labeling  on the path $P_{n_2+1}$ is $(A',B',r-1)$-uniformly ordered, where $\{A',B'\}$ is the vertex bipartition of $P_{n_2+1}$ such that the root is an element of $A'$. This proves the statement in the case $m\equiv 2\mod 4$.
\end{proof}

As a consequence we have:

\begin{corollary}
 For any $m,n_1,n_2\in\mathbb N$, with $m\ge 4$ even, there exists a cyclic $C_{m,(n_1,n_2)}$-decomposition of $K_v$ for any $v\equiv 1 \mod 2(m+n_1+n_2)$.
\end{corollary}

\section{Merging labelings: graphs with more than one cycle}

In this section we are going to use in a more explicit way the merging technique introduced in the previous section to prove the following:

\begin{thm}\label{T:5}
  Let $t\in\mathbb N$ and $m_1,\dots,m_t,n_1,\dots,n_t\in\mathbb N$, such that $m_t\equiv 0\mod 2$, $m_i\equiv 0\mod 4$ and $n_i\equiv 0\mod 2$ for $t\ge 2$ and $i=1,\dots,t-1$. Then there exists a uniformly ordered labeling of $C_{m_1,(n_1),m_2,(n_2),\dots,m_t,(n_t)}$.
\end{thm}
\begin{proof}
  We are going to prove the statement by iteration, the case $t=1$ being proved in Theorem \ref{T:1}. So we suppose that $t\ge 2$.

  \textbf{Case 1.} Suppose that $m_t\equiv 0\mod 4$ and $n_t$ is even. Let:
  \begin{itemize}
\item $\{A_i,B_i\}$ the bipartition of $V(C_{m_1,(n_1),m_2,(n_2),\dots,m_i,(n_i)})$ such that the roots are elements of $A_i$, for any $i=1,\dots,t$
    \item $r_i=2(m_1+\dots+m_i+n_1+\dots+n_i)+1$ for $i=1,\dots,t$
          \item $s_i=\tfrac{m_i+n_i}{2}$ for $i=1,\dots,t$
  \item $p^{(1)}=0$ and $p^{(i)}=m_1+\dots+m_{i-1}+n_1+\dots+n_{i-1}$ for $i=2,\dots,t$
    \item $q^{(i)}=\tfrac{p^{(i)}}{2}+i-1$ for $i=1,\dots,t$
\end{itemize}
and let for $i=1,\dots,t$
\[
  a^{(i)}_j= \begin{cases}
    p^{(i)}+j & \text{for }j=1,\dots,\dfrac{m_i}{2}-1\\[2ex]
    p^{(i)}+j+1 & \text{for }j=\dfrac{m_i}{2},\dots,m_i+n_i-1.
  \end{cases}
\]
For any $i=1,\dots,t$, let:
\[
  x^{(i)}_j=\begin{cases}
      a^{(i)}_1 & \text{for }j=1\\[2ex]
      \displaystyle{\sum_{h=0}^{j-1}a^{(i)}_{2h+1} - \sum_{h=1}^{j-1}a^{(i)}_{2h}} & \text{for }j=2,\dots,s_i
    \end{cases}
\]
and
  \[
  y^{(i)}_j(r)=\sum_{h=0}^{j-1} a^{(i)}_{2h+1}-\sum_{h=1}^j a^{(i)}_{2h}+r,
\]
for $j=1,\dots,s_i-1$. Note that they are defined as in Corollary \ref{C:2}, with respect to the differences $a^{(i)}_1,\dots,a^{(i)}_{m_i+n_i-1}$.

Keeping in mind the proof of Theorem \ref{T:1},  we can assign a uniformly $(A_1,B_1,r_1-1)$-ordered labeling $g\colon V(C_{m_1,(n_1)})\rightarrow [0,r_1-1]$ such that:
\begin{itemize}
  \item $g(A_1)=\{x_j^{(1)}-q^{(1)}-1\mid j=1,\dots,s_1\}$
  \item $g(B_1)=\{y_j^{(1)}(r_1)-q^{(1)}-1\mid j=1,\dots,s_1-1\}\cup\{r_1-q^{(1)}-1\}$
\end{itemize}
where:
\[
  \max g(A_1)=x^{(1)}_{s_1}-q^{(1)}-1 \quad\text{and}\quad \min g(B_1)=y^{(1)}_{s_1-1}(r_1)-q^{(1)}-1.
\]
Since we are going to prove the statement by iteration, we can suppose that for some $i\in\{2,\dots,t-1\}$ it is possible to assign a uniformly $(A_i,B_i,r_i-1)$-ordered labeling $g\colon V(C_{m_1,(n),m_2,(n_2),\dots,m_{i},(n_i)})\rightarrow [0,r_i-1]$ such that:
\begin{itemize}
  \item $g(A_i)=\bigcup_{h=1}^i\{x_j^{(h)}-q^{(h)}-1\mid\, j=1,\dots,s_h\}$
  \item $g(B_i)=\bigcup_{h=1}^i\{y_j^{(h)}(r_i)-q^{(h)}-1\mid\, j=1,\dots,s_h-1\}\cup\{r_i-q^{(h)}-1\mid\, h=1,\dots,i\}$
\end{itemize}
where:
\[
  \max g(A_i)=x^{(i)}_{s_i}-q^{(i)}-1 \quad\text{and}\quad \min g(B_i)=y^{(i)}_{s_i-1}(r_i)-q^{(i)}-1.
\]
By Remark \ref{R:1} we can replace $r_i$, which we use in the definition of all the $y^{(h)}_j(r_i)$ and $r_i-q^{(h)}-1$, with $r_{i+1}$, so that now we have $y^{(h)}_j(r_{i+1})$ for all $h=1,\dots,i$ and $j=1,\dots,s_h-1$ and $r_{i+1}-q^{(h)}-1$ for all $h=1,\dots,i$. In this way, now we have a uniformly $(A_i,B_i,r_{i+1})$-ordered labeling $g'\colon V(C_{m_1,(n),m_2,(n_2),\dots,m_{i},(n_i)})\rightarrow [0,r_{i+1}-1]$ such that:
\begin{itemize}
  \item $g'(A_i)=\bigcup_{h=1}^i\{x_j^{(h)}-q^{(h)}-1\mid\, j=1,\dots,s_h\}$
  \item $g'(B_i)=\bigcup_{h=1}^i\{y_j^{(h)}(r_{i+1})-q^{(h)}-1\mid\,  j=1,\dots,s_h-1\}\cup\{r_{i+1}-q^{(h)}-1\mid\, h=1,\dots,i\}$
\end{itemize}
where:
\[
  \max g'(A)=x^{(i)}_{s_i}-q^{(i)}-1 \quad\text{and}\quad \min g'(B)=y^{(i)}_{s_i-1}(r_{i+1})-q^{(i)}-1.
\]
We can extend the labeling $g'$ to a labeling
\[
  g''\colon V(C_{m_1,(n_1),m_2,(n_2),\dots,m_{i},(n_i),m_{i+1},(n_{i+1})})\rightarrow [0,r_{i+1}-1]
  \]
  in such a way that:
\begin{itemize}
  \item the vertices of the cycle $C_{m_{i+1}}$ and the path $P_{n_{i+1}+1}$ receive the labels $x^{(i+1)}_j-q^{(i+1)}-1$ for $j=1,\dots,s_{i+1}$, $y^{(i+1)}_j(r_{i+1})-q^{(i+1)}-1$ for $j=1,\dots,s_{i+1}-1$ and $r_{i+1}-q^{(i+1)}-1$
  \item the labelings are assigned in such a way that the restriction of this labeling to $V(C_{m_{i+1},(n_{i+1})})$ is a $(r_{i+1}-q^{(i+1)})$-shift of the labeling given in Theorem \ref{T:1}
        \item the vertex (root) common to the path $P_{n_i+1}$ and the cycle $C_{m_{i+1}}$ maintains the labeling $x^{(i)}_{s_i}-q^{(i)}-1=x^{(i+1)}_1-q^{(i+1)}-1$ (where $x_{s_i}^{(i)}=p^{(i)}+s_i$).
\end{itemize}
Note that:
\begin{itemize}
  \item $r_{i+1}-q^{(i+1)}-1<y^{(i)}_{s_i-1}(r_{i+1})-q^{(i)}-1=\min g'(B_i)$
  \item $x^{(i+1)}_1-q^{(i+1)}-1<\dots<x^{(i+1)}_{s_{i+1}}-q^{(i+1)}-1<y^{(i+1)}_{s_{i+1}-1}(r_{i+1})-q^{(i+1)}-1<\dots<r_{i+1}-q^{(i+1)}-1$.
\end{itemize}
So, $g''$ is a uniformly $(A_{i+1},B_{i+1},r_{i+1})$-ordered labeling
\[
  g''\colon V(C_{m_1,(n_1),m_2,(n_2),\dots,m_i,(n_i),m_{{i+1}},(n_{i+1})})\rightarrow [0,r_{i+1}-1]
\]
  such that:
\begin{itemize}
  \item $g''(A_{i+1})=\bigcup_{h=1}^{i+1}\{x_j^{(h)}-q^{(h)}-1\mid\, j=1,\dots,s_h\}$
  \item $g''(B_{i+1})=\bigcup_{h=1}^{i+1}\{y_j^{(h)}(r_{i+1})-q^{(h)}-1\mid\, j=1,\dots,s_h-1\}\cup\{r_{i+1}-q^{(h)}-1\mid\, h=1,\dots,i+1\}$
\end{itemize}
where:
\[
  \max g''(A_{i+1})=x^{(i+1)}_{s_{i+1}}-q^{(i+1)}-1
\]
and
\[
\min g''(B_{i+1})=y^{(i+1)}_{s_{i+1}-1}(r_{i+1})-q^{(i+1)}-1.
\]
It is possible to iterate this procedure and we get the statement for $i=t-1$.

\textbf{Case 2.} If $m_t\equiv 0\mod 4$ and $n_t$ is odd, the previous merging construction still works, the only differences being that the pendant vertex of $P_{n_t+1}$ is an element of $B_t$ and $ \min g''(B_{t})=y^{(t)}_{s_{t}-1}-q^{(t)}-1$.

\textbf{Case 3.} Let $m_t\equiv 2\mod 4$ and $n_t\ge 0$. By case 1, we simply need a merging at the last step. In this case, the only difference with case 1 is that we take for $n_t\ge 1$:
  \[
    a_j^{(t)}=\begin{cases}
      p^{(t)}+j & \text{for } j=1,\dots,\dfrac{m_t}{2}\\
      p^{(t)}+j+1 & \text{for } j=\dfrac{m_t}{2}+1,\dots,m_t-2\\
      p^{(t)}+j+2 & \text{for } j=m_t-1,\dots,m_t+n_t-2\\
      r_t-p^{(t)}-m_t & \text{for } j=m_t+n_t-1
      \end{cases}
    \]
    and for $n_t=0$:
      \[
    a_j^{(t)}=\begin{cases}
      p^{(t)}+j & \text{for } j=1,\dots,\dfrac{m_t}{2}\\
      p^{(t)}+j+1 & \text{for } j=\dfrac{m_t}{2}+1,\dots,m_t-2\\
      p^{(t)}+m_t+1 & \text{for } j=m_t-1.
      \end{cases}
    \]
  \end{proof}

  Again, as a consequence we immediately get:

\begin{corollary}
  Let $t\in\mathbb N$ and $m_1,\dots,m_t,n_1,\dots,n_t\in\mathbb N$, such that $m_t\equiv 0\mod 2$, $m_i\equiv 0\mod 4$ and $n_i\equiv 0\mod 2$ for $t\ge 2$ and $i=1,\dots,t-1$. Then there for any $v\equiv 1\mod (2(m_1+\dots+m_t+n_1+\dots+n_t))$ exists a cyclic $G$-decomposition of $K_v$, where $G=C_{m_1,(n_1),m_2,(n_2),\dots,m_t,(n_t)}$.
\end{corollary}

\end{document}